\documentclass[11pt,reqno]{amsart}

\usepackage[T1]{fontenc}
\usepackage{lmodern}
\usepackage{amsmath,amssymb,amsthm,mathtools}
\usepackage{microtype}
\usepackage[hidelinks,hypertexnames=false]{hyperref}
\usepackage{url}

\allowdisplaybreaks
\newtheorem{theorem}{Theorem}[section]
\newtheorem{proposition}[theorem]{Proposition}
\newtheorem{lemma}[theorem]{Lemma}
\newtheorem{corollary}[theorem]{Corollary}
\theoremstyle{definition}

\theoremstyle{remark}
\newtheorem{remark}[theorem]{Remark}

\newcommand{\R}{\mathbb R}
\newcommand{\Pf}{\operatorname{Pf}}
\newcommand{\supp}{\operatorname{supp}}
\newcommand{\TN}{\mathrm{TN}}
\newcommand{\one}{\mathbf 1}
\newcommand{\cB}{\mathcal B}
\newcommand{\cM}{\mathcal M}

\title[Colombo's determinant problem]
  {Colombo's Determinant Problem}
\author{Qianli Ma}
\address{College of Computer Science and Technology, Zhejiang University,
  Hangzhou, China}
\address{WuJie AI, Hangzhou, China}
\email{qianli.ma@zju.edu.cn}
\thanks{ORCID: 0009-0004-8256-0999.}
\date{Draft for expert circulation, August 17, 2026}


\begin{document}

\begin{abstract}
We completely solve Colombo's 1928 determinant problem.  For distinct real
$x_1,\ldots,x_N$, $N\geq2$, and an integer $D\geq1$,
\[
 \det[(x_j-x_i)^D]\neq0
 \quad\Longleftrightarrow\quad
 D\geq N-1\ \text{and}\ (N\ \text{is even or }D\ \text{is even}).
\]
The even-exponent case follows from Dyn--Goodman--Micchelli (1986); the
remaining odd case is proved by a strict Pfaffian sign theorem.  The new
odd-exponent theorem and its complete proof chain have also been formalized
in Lean~4.
\end{abstract}

\maketitle

\section{Introduction and main theorem}\label{sec:intro}

For integers $N\geq2$ and $D\geq1$, and distinct real numbers
$x_1,\ldots,x_N$, set
\begin{equation}\label{eq:colombo-matrix}
 A_{N,D}(x)=\bigl((x_j-x_i)^D\bigr)_{i,j=1}^{N}.
\end{equation}
Simultaneously reordering the nodes conjugates $A_{N,D}$ by a permutation
matrix, so we may assume $x_1<\cdots<x_N$.  We prove the complete
classification.

\begin{theorem}[Colombo's determinant problem]\label{thm:colombo}
Let $N\geq2$, let $x_1,\ldots,x_N$ be distinct real numbers, and let
$D\geq1$ be an integer.  Then
\begin{equation}\label{eq:complete-classification}
 \det A_{N,D}(x)\neq0
 \quad\Longleftrightarrow\quad
 D\geq N-1\ \text{and}\ (N\ \text{is even or }D\ \text{is even}).
\end{equation}
\end{theorem}

The matrix arose from a boundary-value problem for hyperbolic partial
differential equations, not from matrix theory.  In his 1928 ICM
communication Colombo considered a constant-coefficient strictly hyperbolic
operator whose characteristic directions are real and distinct.  In modern
notation one may write its homogeneous part as
\[
 P=\prod_{s=1}^{N}(\partial_x+\lambda_s\partial_y),
 \qquad \lambda_1,\ldots,\lambda_N\ \text{distinct},
\]
with separated solutions
\[
 u(x,y)=\sum_{s=1}^{N}\varphi_s(y-\lambda_sx).
\]
Restricting to the characteristic lines $y=\lambda_r x$ and comparing the
coefficient of degree $D$ produces exactly the matrix
$((\lambda_r-\lambda_s)^D)_{r,s}$.  Colombo recorded the lower-degree rank
obstruction, the nonsingular endpoint $D=N-1$, and the parity phenomena, and
conjectured that for even $N$ the determinant is nonzero for every
$D\geq N-1$; he verified the cases $N=4,6$ by special arguments
\cite[pp.~35--36]{Colombo1930}.

The later history split into two rather different directions.  Colombo
returned to fourth-order equations in 1932 and 1935
\cite{Colombo1932,Colombo1935}.  Sj\"ostrand's 1929 thesis developed
arbitrary-order Goursat theory for a different common-point,
noncharacteristic configuration \cite{Sjostrand1929}, while Rosati's 1966
paper treated a third-order characteristic polygon problem and described the
older Colombo--Sj\"ostrand line as part of a broader boundary-value program
\cite{Rosati1966}.  In this PDE literature, the power-difference determinant
itself does not appear to have become a standard object.  A plausible
historical explanation is that the PDE questions moved toward other boundary
geometries and compatibility problems, whereas the tools naturally suited to
\eqref{eq:colombo-matrix} were being developed in approximation theory and
matrix positivity.  This separation of communities may help explain why the
original determinant problem slipped from view.

Meanwhile those tools matured.  Schoenberg's 1946 work laid foundations for
modern spline theory \cite{Schoenberg1946}; Schoenberg--Whitney connected
spline interpolation with positivity of translation determinants in 1953
\cite{SchoenbergWhitney1953}; de Bruijn's Pfaffian integration formula
appeared in 1955 \cite{deBruijn1955}; Marsden's reproduction identity in
1970 \cite{Marsden1970}; and Karlin systematically linked total positivity
and spline interpolation in 1971 \cite{Karlin1971}.  By the mid-1980s the
classical toolkit used below was essentially mature: de Boor--DeVore gave a
particularly transparent proof of total positivity for B-spline collocation
matrices in 1985 \cite{deBoorDeVore1985}.  The two-fan coefficient matrix and
the support-aware recognition step required below are not supplied by these
classical results, but the surrounding language and positivity machinery were
already in place.

There is also a second, previously separate, route into Colombo's problem.
For
\[
 B_\alpha=\bigl(|x_i-x_j|^\alpha\bigr)_{i,j=1}^{N},
\]
Dyn, Goodman, and Micchelli computed its singularity and inertia behavior in
1986 \cite{DynGoodmanMicchelli1986}; see also the modern summary
\cite[\S6]{BhatiaJain2015}.  In particular, $B_\alpha$ is singular exactly
when $\alpha$ is an even integer smaller than $N-1$.  Since translating all nodes by the same constant does not change $B_\alpha$, we
may place the nodes in the positive half-line if required by the formulation of
the cited theorem.  Since $A_{N,D}=B_D$ for even $D$, their theorem already
settles the complete even-exponent branch of Colombo's determinant problem.  What remained is the
case of even matrix size and odd exponent.

We use the convention
\[
 \Pf\begin{pmatrix}0&a\\-a&0\end{pmatrix}=a,
 \qquad A_{ij}=(x_j-x_i)^D.
\]
The new part of the paper is the following strict sign theorem.

\begin{theorem}[Odd-exponent Colombo theorem]\label{thm:main}
Let $m\geq1$, let $x_1<\cdots<x_{2m}$ be real, and let $r\geq m-1$ be
an integer.  Put $D=2r+1$.  Then
\begin{equation}\label{eq:main-sign}
 (-1)^{\binom m2}\Pf\bigl[(x_j-x_i)^{2r+1}\bigr]_{i,j=1}^{2m}>0.
\end{equation}
Consequently,
\begin{equation}\label{eq:main-det}
 \det\bigl[(x_j-x_i)^{2r+1}\bigr]_{i,j=1}^{2m}>0.
\end{equation}
\end{theorem}

At the minimal exponent $D=N-1$, this is the classical Torelli Pfaffian
identity; a modern normalization appears in
\cite[Corollary~5.2]{EhrenborgFox2014}.  Theorem~\ref{thm:main} covers every
larger odd exponent as well.

\begin{proof}[Proof of Theorem~\ref{thm:colombo}]
If $D<N-1$, then $\operatorname{rank}A_{N,D}\leq D+1<N$, since the
columns are evaluations of polynomials of degree at most $D$.  If both $N$
and $D$ are odd, then $A_{N,D}$ is an odd-dimensional skew-symmetric matrix
and is singular.  If $D$ is even, then $A_{N,D}=B_D$, and the theorem of
Dyn--Goodman--Micchelli cited above says that it is nonsingular exactly when
$D\geq N-1$.  The only remaining case is $N=2m$ and odd
$D\geq N-1$, which is Theorem~\ref{thm:main}.
\end{proof}

For $m=1$, the Pfaffian is the single positive entry
$(x_2-x_1)^{2r+1}$.  Assume henceforth that $m\geq2$.  The proof of the new
odd-exponent theorem has four load-bearing steps.

\begin{enumerate}
\item At the critical degree $d=m-2$, a paired determinant of left and
right truncated powers is nonnegative.  Marsden reproduction reduces this
to total nonnegativity of a rectangular two-fan coefficient matrix.
\item Its supported solid minors are evaluated by a sliding-root product at
small order and by subtraction-free Desnanot--Jacobi condensation at large
order.  A monotone-interval staircase lemma promotes the solid signs to all
minors without assuming total nonnegativity.
\item Exact Volterra convolution propagates the paired inequality from
degree $m-2$ to every higher integer degree.
\item A Beta integral and de Bruijn's Pfaffian identity express the signed
Pfaffian as an integral of those paired determinants.  An explicit central
simplex makes the integral strictly positive.
\end{enumerate}

Classical total-positivity results treat one-sided augmented truncated-power
kernels and individual B-spline systems; see, for example,
\cite{Karlin1971,Bojanov1990}.  They do not directly cover the two oppositely
oriented moving fans under the coordinatewise pairing used below.  Likewise,
the dispersion--Pl\"ucker mechanism has established antecedents in the
literature on almost strictly totally positive matrices
\cite{GascaMicchelliPena1992,Gladwell2004,AlonsoPenaSerrano2022}.  Those
results do not supply a noncircular shortcut: in particular, the 1992 ASTP
definition assumes total nonnegativity at the outset.  We therefore include
the support-aware recognition argument in full.

The paper is organized along the four steps above.  The discontinuous
degree-zero boundary $m=2$ is proved separately in Appendix~\ref{app:m2},
and Appendix~\ref{app:ledger} collects all constants and signs.

\section{Paired split determinants}\label{sec:split}

For $y\in\R$ and an integer $p\geq0$, write
\[
 y_+^p=\begin{cases}y^p,&y>0,\\0,&y\leq0.\end{cases}
\]
Thus $y_+^0=\one_{y>0}$.  Let
\[
 X=(x_1<\cdots<x_{2m}),\quad
 S=(s_1<\cdots<s_m),\quad
 U=(u_1<\cdots<u_m),
\]
and assume
\begin{equation}\label{eq:pair-condition}
 s_j<u_j\qquad(1\leq j\leq m).
\end{equation}
Define
\begin{equation}\label{eq:Hp}
 H_p(X;S,U)=
 \det\left[
  (s_j-x_i)_+^p\ \middle|\ (x_i-u_j)_+^p
 \right]_{\substack{1\leq i\leq2m\\1\leq j\leq m}}.
\end{equation}

\begin{theorem}[Paired split-determinant inequality]\label{thm:split}
For $m\geq2$ and every integer $p\geq m-2$,
\begin{equation}\label{eq:Hp-nonnegative}
 H_p(X;S,U)\geq0.
\end{equation}
The assertion includes arbitrary coincidences between the two lists $S$
and $U$ that are compatible with \eqref{eq:pair-condition}, and sampled
points may lie on knots.
\end{theorem}

The critical case is the main work.

\begin{proposition}[Critical degree]\label{prop:critical}
For $m\geq2$,
\begin{equation}\label{eq:critical}
 H_{m-2}(X;S,U)\geq0.
\end{equation}
\end{proposition}

Sections~\ref{sec:marsden}--\ref{sec:staircase} prove
Proposition~\ref{prop:critical}.  Section~\ref{sec:volterra} then proves
Theorem~\ref{thm:split}.  The degree $m-2$ is sharp for a uniform assertion;
explicit lower-degree counterexamples are recorded in
Remark~\ref{rem:sharpness}.

\section{Marsden common refinement}\label{sec:marsden}

In this section assume first that $m\geq3$ and that all knots in $S\cup U$
are distinct.  Put
\begin{equation}\label{eq:d}
 d=m-2.
\end{equation}
Merge the two lists as
\[
 z_0<z_1<\cdots<z_{2m-1},
\]
and let $\alpha_j$ and $\beta_j$ be the zero-based ranks of $s_j$ and
$u_j$.  The paired inequalities imply the Dyck conditions
\begin{equation}\label{eq:dyck-ranks}
 \alpha_j<\beta_j,\qquad \alpha_1=0,\qquad \beta_m=2m-1.
\end{equation}

Choose exterior anchors
\begin{equation}\label{eq:anchors}
 A<\min(X,S,U),\qquad B>\max(X,S,U),
\end{equation}
and form the open knot vector
\begin{equation}\label{eq:knot-vector}
 T=(A^{d+1},z_0,\ldots,z_{2m-1},B^{d+1}).
\end{equation}
We index its entries from zero.  Let $B_{i,d}$, $0\leq i\leq3m-2$, be
the normalized degree-$d$ B-splines on $T$, and let
\begin{equation}\label{eq:BT}
 \cB_T(X)=\bigl(B_{j,d}(x_i)\bigr)_
 {\substack{1\leq i\leq2m\\0\leq j\leq3m-2}}
\end{equation}
be their collocation matrix.  Normalized Marsden reproduction
\cite{Marsden1970} gives the exact factorization
\begin{equation}\label{eq:marsden-factor}
 \cM_d(X;S,U)=\cB_T(X)C,
\end{equation}
where $\cM_d$ is the split matrix in \eqref{eq:Hp}, and
\[
 C=[L_1,\ldots,L_m,R_1,\ldots,R_m]
\]
is a $(3m-1)\times2m$ coefficient matrix with
\begin{align}
 C_{i,L_j}
 &=\one_{i\leq\alpha_j}
   \prod_{h=1}^{d}(s_j-T_{i+h}),                         \label{eq:C-left}\\
 C_{i,R_j}
 &=\one_{i\geq m-1+\beta_j}
   \prod_{h=1}^{d}(T_{i+h}-u_j).                        \label{eq:C-right}
\end{align}
The normalization scalar in \eqref{eq:marsden-factor} is one.  For the
right block, the reflection
$(x-u_j)^d=(-1)^d(u_j-x)^d$ is exactly canceled by reversing every factor
in the Marsden coefficient; no residual sign remains.

Every displayed nonzero entry is positive.  In particular,
\begin{equation}\label{eq:C-supports}
 \supp L_j=[0,\alpha_j],\qquad
 \supp R_j=[m-1+\beta_j,3m-2].
\end{equation}
Both the left and right endpoints of these column intervals are
nondecreasing in the natural column order.  We shall prove
\begin{equation}\label{eq:C-TN-goal}
 C\in\TN.
\end{equation}

The ordinary B-spline collocation matrix \eqref{eq:BT} is totally
nonnegative \cite[Theorem~2]{deBoorDeVore1985}.  Hence
\eqref{eq:C-TN-goal}, \eqref{eq:marsden-factor}, and Cauchy--Binet imply
\eqref{eq:critical}.  Cross-family knot coincidences follow by continuity
because $d\geq1$.  The case $m=2,d=0$, where such a continuity argument is
unavailable, is Appendix~\ref{app:m2}.

\section{Supported solid minors of the coefficient matrix}
\label{sec:solid}

We index the rows of $C$ by $0,\ldots,3m-2$.  A solid minor uses consecutive
rows and consecutive columns.  One-sided solid minors lie wholly in the
$L$ or $R$ block.  A seam minor contains a suffix of the $L$ block and a
prefix of the $R$ block.  It is convenient to prove a slightly stronger
statement allowing an arbitrary consecutive right block.

Fix columns
\begin{equation}\label{eq:seam-columns}
 L_a,\ldots,L_m,R_c,\ldots,R_h,\qquad h=c+b-1,
\end{equation}
put
\begin{equation}\label{eq:seam-q}
 \ell=m-a+1,\qquad q=\ell+b,
\end{equation}
and take rows $\rho,\ldots,\rho+q-1$.  Denote the resulting determinant by
$D(a,c,b;\rho)$.

\begin{lemma}[Support interval]\label{lem:support-interval}
The determinant $D(a,c,b;\rho)$ has a supported permutation if and only if
\begin{equation}\label{eq:support-interval}
 m+\beta_h-q\leq\rho\leq\alpha_a.
\end{equation}
If this interval condition fails, the determinant is zero.
\end{lemma}

\begin{proof}
Because the $L$ columns have prefix supports and the $R$ columns have suffix
supports, any supported matching can be uncrossed.  The canonical matching
pairs the $L$ columns with the first $\ell$ selected rows and the $R$
columns with the last $b$ rows.  Monotonicity of the merge ranks reduces
Hall's conditions for this matching to the last-right and first-left endpoint
tests in \eqref{eq:support-interval}.  If either fails, no determinant term
is supported.
\end{proof}

We first handle orders at most $m-1=d+1$.

\begin{lemma}[Small seam product]\label{lem:small-seam}
Assume $q\leq m-1$ and \eqref{eq:support-interval}.  For $0\leq v<q$ set
\begin{align}
 F_v(y)&=\prod_{h=1}^{d}(y-T_{\rho+v+h}),                \label{eq:Fv}\\
 G_{\rho,q}(y)&=\prod_{h=q}^{d}(y-T_{\rho+h}),          \label{eq:G}\\
 \phi_v(y)&=
 \prod_{h=v+1}^{q-1}(y-T_{\rho+h})
 \prod_{h=0}^{v-1}(y-T_{\rho+d+1+h}).                  \label{eq:phi}
\end{align}
Then $F_v=G_{\rho,q}\phi_v$, and
\begin{equation}\label{eq:phi-evaluation}
 \det[\phi_v(y_j)]_{v,j=0}^{q-1}
 =K_{\rho,q}\prod_{i<j}(y_j-y_i),
\end{equation}
where
\begin{equation}\label{eq:K}
 K_{\rho,q}=
 \prod_{0\leq i<j\leq q-1}
 (T_{\rho+d+1+i}-T_{\rho+j})>0.
\end{equation}
Consequently, with
$\Delta(y_1,\ldots,y_k)=\prod_{i<j}(y_j-y_i)$,
\begin{align}
 D(a,c,b;\rho)
 &=K_{\rho,q}
   \prod_{j=a}^{m}G_{\rho,q}(s_j)
   \prod_{j=c}^{h}|G_{\rho,q}(u_j)|                    \notag\\
 &\quad\cdot\Delta(s_a,\ldots,s_m)
   \Delta(u_c,\ldots,u_h)
   \prod_{\substack{a\leq i\leq m\\c\leq j\leq h}}
   (s_i-u_j)>0.                                         \label{eq:small-product}
\end{align}
\end{lemma}

\begin{proof}
The common factors in \eqref{eq:Fv} are exactly
\eqref{eq:G}; the remaining roots give \eqref{eq:phi}.  Evaluating the
$\phi_v$ successively at the early roots
$T_{\rho+1},\ldots,T_{\rho+q-1}$ triangularizes the coefficient matrix and
gives \eqref{eq:K}.  Since both sides of \eqref{eq:phi-evaluation} are
alternating polynomials of the same degree, this proves
\eqref{eq:phi-evaluation}.

On the support interval, the masks in \eqref{eq:C-left}--\eqref{eq:C-right}
agree with the full sliding polynomials: a nominally masked evaluation is
already zero at one of the roots in \eqref{eq:Fv}.  The same condition also
orders the selected points as
\[
 u_c<\cdots<u_h<s_a<\cdots<s_m.
\]
Extracting the common factors and applying
\eqref{eq:phi-evaluation} gives \eqref{eq:small-product} up to sign.  Each
right column contributes $(-1)^d$; its common factor at $u_j$ has sign
$(-1)^{d-q+1}$, leaving $(-1)^{q-1}$ per right column.  Moving the $b$
right points before the $\ell$ left points contributes $(-1)^{b\ell}$.
The total exponent is
\[
 b(q-1)+b\ell=b(b-1)+2b\ell,
\]
which is even.
\end{proof}

The same calculation, with one point list absent, proves positivity for
every supported consecutive one-sided minor of order at most $m-1$.  The
only larger one-sided case has order $m$; expansion in $L_1$, respectively
$R_m$, reduces it with positive sign to order $m-1$.

We now treat all remaining seam minors.

\begin{lemma}[Width one]\label{lem:width-one}
Assume $q\geq m$.  The support interval in
\eqref{eq:support-interval}, written $[L,U]$, has width at most one:
\begin{equation}\label{eq:LU}
 L=m+\beta_h-q,\qquad U=\alpha_a,\qquad U-L\leq1.
\end{equation}
\end{lemma}

\begin{proof}
The inequality $q=m-a+1+b\geq m$ gives $b\geq a-1$, hence
$h=c+b-1\geq a-1$.  The sequence $\beta_j-j$ is nondecreasing.  If
$a\geq2$, the Dyck property gives
\begin{equation}\label{eq:alpha-beta-adj}
 \alpha_a\leq\beta_{a-1}+1.
\end{equation}
Indeed, if $u_{a-1}$ occurs after $s_a$, this is immediate; if it occurs
before $s_a$, the merge path is at height zero just after $u_{a-1}$, so its
next event must be $s_a$.  Moreover,
\begin{align*}
 L&=a-1+\beta_h-b
   =a+c-2+(\beta_h-h)\\
  &\geq a+c-2+\beta_{a-1}-(a-1)
   =\beta_{a-1}+c-1\geq\beta_{a-1}.
\end{align*}
Together with \eqref{eq:alpha-beta-adj}, this proves $U-L\leq1$.  If
$a=1$, then $U=\alpha_1=0$ and $L=\beta_h-b\geq h-b=c-1\geq0$.
\end{proof}

\begin{lemma}[Large seam condensation]\label{lem:large-seam}
Every supported seam minor of order $q\geq m$ is strictly positive.
\end{lemma}

\begin{proof}
Apply Desnanot--Jacobi to $D=D(a,c,b;\rho)$, deleting the first and last
rows and columns.  We extend the notation by letting $D(a,c,0;\rho)$
denote the corresponding left-only solid minor and
$D(m+1,c,b;\rho)$ the corresponding right-only solid minor; the empty
determinant is $1$.  Then
\begin{equation}\label{eq:DJ}
 D D_{\mathrm{mid}}
 =D_{\mathrm{TL}}D_{\mathrm{BR}}
  -D_{\mathrm{TR}}D_{\mathrm{BL}},
\end{equation}
where
\begin{align*}
 D_{\mathrm{TL}}&=D(a,c,b-1;\rho),&
 D_{\mathrm{BR}}&=D(a+1,c,b;\rho+1),\\
 D_{\mathrm{TR}}&=D(a+1,c,b;\rho),&
 D_{\mathrm{BL}}&=D(a,c,b-1;\rho+1),\\
 D_{\mathrm{mid}}&=D(a+1,c,b-1;\rho+1).
\end{align*}
For every $\rho\in[L,U]$, the three minors
$D_{\mathrm{TL}},D_{\mathrm{BR}},D_{\mathrm{mid}}$ satisfy their support
conditions and are positive by lower-order induction.  If $\rho=L$, then
$D_{\mathrm{TR}}=0$; if $\rho=U=L+1$, then
$D_{\mathrm{BL}}=0$; if $L=U$, both statements apply.  Lemma
\ref{lem:width-one} leaves no other value.  Thus one cross product in
\eqref{eq:DJ} vanishes and
\begin{equation}\label{eq:DJ-positive}
 D=\frac{D_{\mathrm{TL}}D_{\mathrm{BR}}}
          {D_{\mathrm{mid}}}>0.
\end{equation}
The induction starts at $q=m$, where all three surviving minors have order
at most $m-1$ and are covered by Lemma~\ref{lem:small-seam} or its one-sided
boundary.  If $b=1$ or $\ell=1$, the corresponding neighbor is one-sided;
the order-zero middle determinant is interpreted as $1$.
\end{proof}

Combining the preceding lemmas yields the local statement needed later.

\begin{proposition}[Solid-minor theorem]\label{prop:solid}
Every solid minor of $C$ is nonnegative.  It is positive exactly when its
sorted diagonal lies in the interval supports \eqref{eq:C-supports}.
\end{proposition}

\section{From solid minors to total nonnegativity}\label{sec:staircase}

We isolate the recognition principle used for the rectangular matrix $C$.
The dispersion induction, three-term Pl\"ucker identity, and zero-shadow
factorization are established mechanisms in the ASTP/ASSR literature; see
Gasca--Micchelli--Pe\~na
\cite[Definition~3.1 and Theorem~3.1]{GascaMicchelliPena1992}, Gladwell
\cite[Theorem~2.1]{Gladwell2004}, and Alonso--Pe\~na--Serrano
\cite[Theorem~1]{AlonsoPenaSerrano2022}.  We prove the precise form required
here because the first theorem assumes total nonnegativity, while the latter
two impose square/nonsingular or maximal-rank supported-diagonal hypotheses
that the raw coefficient matrix need not satisfy.

\begin{theorem}[Monotone-interval staircase recognition]
\label{thm:staircase}
Let $M$ be a real rectangular matrix whose $j$th column is strictly positive
exactly on an integer interval $[\lambda_j,\rho_j]$.  Assume both endpoint
sequences are nondecreasing.  If every solid submatrix whose sorted diagonal
is supported has positive determinant, then every minor of $M$ is
nonnegative.  More precisely, a minor is positive exactly when its sorted
diagonal is supported.  Maximal blocks of empty rows are allowed and act as
separators.
\end{theorem}

\begin{proof}
First, a supported matching can be uncrossed.  If two increasing columns are
matched to rows in the opposite order, monotonicity of both interval
endpoints permits the two matches to be exchanged.  Hence a selected
submatrix has a supported determinant term if and only if its sorted
diagonal is supported.  If the sorted diagonal is unsupported, every term is
zero.

For a supported minor, write
\[
 \Delta(I,J)=\det M[I,J],
\]
and induct simultaneously on its order and on the sum of its row and column
dispersions.  Solid minors are the base.  Suppose first that
$J=\{j_1\}\cup\tau\cup\{j_k\}$ has a gap, choose a missing column
$p\in(j_1,j_k)$, and let $w$ be $I$ with its last row removed.  If a
superdiagonal entry in $M[I,J]$ vanishes, monotone interval support forces a
whole upper-right zero rectangle, so the determinant factors, with positive
sign, into two smaller supported diagonal blocks.

Otherwise the three-term Pl\"ucker identity is
\begin{align}
 &\Delta(w,\tau\cup\{p\})
  \Delta(I,\tau\cup\{j_1,j_k\})                         \notag\\
 &=\Delta(w,\tau\cup\{j_k\})
   \Delta(I,\tau\cup\{j_1,p\})                         \notag\\
 &\quad+
   \Delta(w,\tau\cup\{j_1\})
   \Delta(I,\tau\cup\{p,j_k\}).                      \label{eq:plucker}
\end{align}
The sorted diagonals of
\[
 \Delta(w,\tau\cup\{p\}),\qquad
 \Delta(w,\tau\cup\{j_1\}),\qquad
 \Delta(I,\tau\cup\{p,j_k\})
\]
are supported.  The first is therefore a positive order-$(k-1)$ divisor.
The product of the latter two factors is strictly positive by induction,
whereas the other product on the right of \eqref{eq:plucker} is
nonnegative.  Both order-$k$ minors on the right have smaller column
dispersion.  Dividing by the first factor proves that the target minor is
positive.  Notice that the selected row set was arbitrary, so simultaneous
row and column gaps have already been handled.

It remains only the case in which the columns are consecutive and the rows
have a gap.  For each row of $M$, the columns containing that row form an
interval whose endpoints move weakly forward with the row index.  Away from
empty rows, the transpose therefore has the same monotone-interval property,
and the preceding argument reduces row dispersion.

Finally, an empty row block determines a column cut: columns supported above
the block precede those supported below it.  A supported selected minor
spanning the block is block diagonal at this cut and factors without a
permutation sign; a minor selecting an empty row is zero.  This proves the
separator assertion and completes the induction.
\end{proof}

\begin{corollary}\label{cor:C-TN}
The Marsden coefficient matrix $C$ in
\eqref{eq:C-left}--\eqref{eq:C-right} is totally nonnegative.
\end{corollary}

\begin{proof}
The endpoints in \eqref{eq:C-supports} are nondecreasing and
Proposition~\ref{prop:solid} supplies the solid-minor hypothesis.  Its only
possible empty rows form the interval
\[
 \alpha_m<i<m-1+\beta_1,
\]
with only $L$ columns supported above and only $R$ columns supported below.
The separator clause of Theorem~\ref{thm:staircase} applies.
\end{proof}

\begin{proof}[Proof of Proposition~\ref{prop:critical} for $m\geq3$]
The matrix $\cB_T(X)$ is totally nonnegative, as is $C$ by
Corollary~\ref{cor:C-TN}.  The factorization
\eqref{eq:marsden-factor} and Cauchy--Binet give
$H_{m-2}(X;S,U)\geq0$ when $S\cup U$ is distinct.  Since $m-2\geq1$, the
general case follows by continuity.  The remaining case $m=2$ is proved in
Appendix~\ref{app:m2}.
\end{proof}

\section{Volterra propagation}\label{sec:volterra}

The critical inequality propagates to every higher degree through an exact
antisymmetrized convolution.

\begin{proposition}[Volterra recurrence]\label{prop:volterra}
For $p\geq1$,
\begin{align}
 H_p(X;S,U)
 &=p^{2m}
 \int_{\substack{a_1<\cdots<a_m\\ b_1<\cdots<b_m}}
 H_{p-1}(X;a,b)                                         \notag\\
 &\qquad\cdot
 \det[\one_{a_i<s_j}]_{i,j=1}^{m}
 \det[\one_{u_j<b_i}]_{i,j=1}^{m}\,da\,db.            \label{eq:volterra}
\end{align}
There is no additional factorial.  Each step determinant is either zero or
$+1$; where both are nonzero, $a_i<s_i<u_i<b_i$ for every $i$.
\end{proposition}

\begin{proof}
The scalar identities
\begin{align}
 (s-x)_+^p
 &=p\int_{\R}(a-x)_+^{p-1}\one_{a<s}\,da,               \label{eq:volt-left}\\
 (x-u)_+^p
 &=p\int_{\R}(x-b)_+^{p-1}\one_{u<b}\,db               \label{eq:volt-right}
\end{align}
are immediate.  Wedge the $m$ left column integrals.  Partition the labelled
$a$-domain into its $m!$ order chambers and rename the variables increasingly
in each chamber.  The wedge permutation sign converts the chamber sum into
\[
 \sum_{\sigma\in S_m}\operatorname{sgn}(\sigma)
 \prod_{j=1}^{m}\one_{a_{\sigma(j)}<s_j}
 =\det[\one_{a_i<s_j}].
\]
Thus the single ordered chamber already contains the complete alternating
sum; it is not multiplied or divided by $m!$.  The right wedge gives the
second determinant and another factor $p^m$.  The left columns remain before
the right columns, so no cross-block sign occurs.  This proves
\eqref{eq:volterra}.

For increasing $a$ and $s$, the first step determinant equals $1$ precisely
on
\begin{equation}\label{eq:left-step-support}
 a_1<s_1<a_2<s_2<\cdots<a_m<s_m,
\end{equation}
up to boundary faces, and vanishes otherwise.  The second equals $1$
precisely on
\begin{equation}\label{eq:right-step-support}
 u_1<b_1<u_2<b_2<\cdots<u_m<b_m.
\end{equation}
Their common support implies $a_i<s_i<u_i<b_i$.
\end{proof}

\begin{proof}[Proof of Theorem~\ref{thm:split}]
Proposition~\ref{prop:critical} is the base $p=m-2$.  If
$H_{p-1}(X;a,b)\geq0$ for every paired $a_i<b_i$, then every factor in the
integrand of \eqref{eq:volterra} is nonnegative.  Hence $H_p\geq0$.  For
$m=2$, the first recurrence used is $p=1$, so no negative exponent is
introduced.
\end{proof}

\section{The Beta--de Bruijn bridge}\label{sec:debruijn}

Fix an integer $r\geq m-1$ and put $D=2r+1$.  For each $t\in\R$, define column
vectors in $\R^{2m}$ by
\begin{equation}\label{eq:LR-vectors}
 L_i(t)=(t-x_i)_+^r,\qquad R_i(t)=(x_i-t)_+^r,
 \qquad 1\leq i\leq2m.
\end{equation}

\begin{lemma}[Beta normalization]\label{lem:beta}
With
\begin{equation}\label{eq:cr}
 c_r=\frac{(2r+1)!}{(r!)^2},
\end{equation}
one has, for all $i,j$,
\begin{equation}\label{eq:beta}
 (x_j-x_i)^{2r+1}
 =c_r\int_{\R}
 \bigl(L_i(t)R_j(t)-L_j(t)R_i(t)\bigr)\,dt.
\end{equation}
\end{lemma}

\begin{proof}
For $i<j$, only the first product is nonzero, and only for
$x_i<t<x_j$.  After affine rescaling, the integral is
\[
 (x_j-x_i)^{2r+1}\int_0^1u^r(1-u)^r\,du
 =\frac{(r!)^2}{(2r+1)!}(x_j-x_i)^{2r+1}.
\]
Skew-symmetry gives the remaining cases.
\end{proof}

We now apply the exterior form of de Bruijn's multiple-integral identity
\cite{deBruijn1955}.

\begin{proposition}[Ordered Pfaffian integral]\label{prop:pf-integral}
Let $A=((x_j-x_i)^{2r+1})_{i,j=1}^{2m}$.  Then
\begin{equation}\label{eq:pf-integral}
 (-1)^{\binom m2}\Pf A
 =c_r^m\int_{t_1<\cdots<t_m}H_r(X;t,t)\,dt.
\end{equation}
\end{proposition}

\begin{proof}
Equation \eqref{eq:beta} says
\[
 A=c_r\int_{\R}\bigl(L(t)R(t)^T-R(t)L(t)^T\bigr)\,dt.
\]
Exterior expansion gives
\begin{align}
 \Pf A
 =\frac{c_r^m}{m!}\int_{\R^m}
 \det[L(t_1),R(t_1),\ldots,L(t_m),R(t_m)]\,dt.           \label{eq:debruijn-full}
\end{align}
Swapping two variables exchanges two blocks of two columns and has sign
$+1$.  The $m!$ order chambers therefore cancel the factor $1/m!$ in
\eqref{eq:debruijn-full}.  Moving each $R(t_j)$ past
$L(t_{j+1}),\ldots,L(t_m)$ requires
\[
 \sum_{j=1}^{m}(m-j)=\binom m2
\]
column swaps.  The grouped determinant is exactly $H_r(X;t,t)$, which proves
\eqref{eq:pf-integral}.  No factor $2^m$ occurs because the orientation
inside each pair is already fixed by $LR^T-RL^T$.
\end{proof}

For $\varepsilon>0$, set
\[
 U^{(\varepsilon)}=(t_1+\varepsilon,\ldots,t_m+\varepsilon).
\]
Theorem~\ref{thm:split} gives
$H_r(X;t,U^{(\varepsilon)})\geq0$.  Since $r\geq m-1\geq1$, truncated
powers of degree $r$ are continuous in their knots; hence, on letting
$\varepsilon\downarrow0$, we obtain $H_r(X;t,t)\geq0$.  Thus the integrand
in \eqref{eq:pf-integral} is nonnegative.  It remains only to prove that it
is not almost everywhere zero.

\section{Strictness and proof of the main theorem}\label{sec:strict}

Consider the open central simplex
\begin{equation}\label{eq:central-simplex}
 x_m<t_1<\cdots<t_m<x_{m+1}.
\end{equation}
On this set the split matrix is block diagonal, so
\begin{equation}\label{eq:block-factor}
 H_r(X;t,t)
 =\det[(t_j-x_i)^r]_{i,j=1}^{m}
  \det[(x_{m+i}-t_j)^r]_{i,j=1}^{m}.
\end{equation}

\begin{lemma}[Separated powers]\label{lem:separated}
If $y_1<\cdots<y_m<z_1<\cdots<z_m$ and $r\geq m-1$, then
\begin{equation}\label{eq:separated-left}
 \det[(z_j-y_i)^r]_{i,j=1}^{m}>0.
\end{equation}
Likewise,
\[
 \det[(w_i-z_j)^r]_{i,j=1}^{m}>0
\]
whenever $z_1<\cdots<z_m<w_1<\cdots<w_m$.
\end{lemma}

\begin{proof}
Choose $c$ strictly between $y_m$ and $z_1$, and put
$Y_i=c-y_i>0$ and $Z_j=z_j-c>0$.  The $Y_i$ decrease while the $Z_j$
increase.  Expand
\begin{equation}\label{eq:shifted-binomial}
 (z_j-y_i)^r=(Y_i+Z_j)^r
 =\sum_{k=0}^{r}\binom rkY_i^kZ_j^{r-k}.
\end{equation}
For every increasing exponent subset in Cauchy--Binet, the generalized
Vandermonde in the decreasing $Y_i$ has sign
$(-1)^{\binom m2}$.  The $Z_j$ determinant has decreasing exponents and the
same sign.  Their product is positive, and all binomial weights are
positive.  The subset $0,1,\ldots,m-1$ occurs because $r\geq m-1$, so the
sum is strict.  The second assertion is the first assertion applied to the
two lists $z_1<\cdots<z_m<w_1<\cdots<w_m$, followed by transposition.
\end{proof}

\begin{proof}[Proof of Theorem~\ref{thm:main}]
Lemma~\ref{lem:separated} and \eqref{eq:block-factor} show that
$H_r(X;t,t)>0$ throughout the open simplex
\eqref{eq:central-simplex}, which has positive measure.  The integrand in
\eqref{eq:pf-integral} is nonnegative everywhere by
Theorem~\ref{thm:split}.  Therefore the integral is strictly positive and
\eqref{eq:main-sign} follows.  Since $A$ is skew-symmetric,
$\det A=(\Pf A)^2>0$, proving \eqref{eq:main-det}.
\end{proof}

\section{Boundary and sharpness remarks}\label{sec:remarks}

\begin{remark}[Coincident knot lists]
The Marsden calculation was performed first with $S\cup U$ distinct.  For
$m\geq3$, the critical degree $m-2$ is at least one, so cross-family
coincidences follow by continuity.  In the Pfaffian application the two lists
are equal, $S=U=t$, but the exponent $r\geq m-1\geq1$, so the same
continuity argument applies.  The discontinuous case $m=2,d=0$ is handled
without B-splines in Appendix~\ref{app:m2}.
\end{remark}

\begin{remark}[Sharpness of the paired threshold]\label{rem:sharpness}
The degree in Theorem~\ref{thm:split} cannot be lowered uniformly.  For
$m=3,p=0=m-3$,
\[
 X=(0,3,5,7,11,14),\quad S=(2,6,8),\quad U=(4,10,12)
\]
gives $H_0=-1$, whereas $H_1=36$.  For $m=4,p=1=m-3$,
\begin{align*}
 X&=(0,13,20,38,60,86,88,100),\\
 S&=(4,50,70,83),\\
 U&=(27,66,78,90)
\end{align*}
gives $H_1=-9{,}609{,}600$.
\end{remark}

\begin{remark}[Scope and provenance]
The new argument of this paper is Theorem~\ref{thm:main}, the odd-exponent
case.  The even-exponent part of Theorem~\ref{thm:colombo} is classical and
follows from Dyn--Goodman--Micchelli \cite{DynGoodmanMicchelli1986}.  The
complete classification is obtained by combining these two inputs with the
elementary low-rank and odd-dimensional skew-symmetry obstructions.  The
historical discussion in the Introduction records the sources located in
preparing this draft; it is not intended as a claim that no equivalent
formulation could occur elsewhere in the older PDE literature.
\end{remark}

\appendix

\section{The degree-zero boundary}\label{app:m2}

We prove Proposition~\ref{prop:critical} when $m=2$.  Let
\[
 p_k=(\one_{i\leq k})_{i=1}^{4},\qquad
 q_k=(\one_{i\geq k})_{i=1}^{4},
\]
and set
\begin{equation}\label{eq:ell-r}
 \ell_j=\#\{i:x_i<s_j\},\qquad
 r_j=1+\#\{i:x_i\leq u_j\}.
\end{equation}
With the strict convention $y_+^0=\one_{y>0}$, the split matrix is
\begin{equation}\label{eq:m2-split}
 [p_{\ell_1},p_{\ell_2},q_{r_1},q_{r_2}],
\end{equation}
where
\begin{equation}\label{eq:m2-order}
 \ell_1\leq\ell_2,\qquad r_1\leq r_2,\qquad \ell_j<r_j.
\end{equation}

Apply the determinant-one row-difference map
\[
 (v_1,v_2,v_3,v_4)\longmapsto
 (v_1-v_2,v_2-v_3,v_3-v_4,v_4).
\]
It sends $p_k$ to $e_k$ and $q_r$ to $e_4-e_{r-1}$, with the natural
zero-vector conventions at the endpoints.  Therefore
\eqref{eq:m2-split} becomes
\begin{equation}\label{eq:m2-basis-det}
 \det[e_{\ell_1},e_{\ell_2},
      e_4-e_{r_1-1},e_4-e_{r_2-1}].
\end{equation}
Zero or repeated columns eliminate all but six index quadruples compatible
with \eqref{eq:m2-order}.  Direct substitution shows that the determinant
is nonzero only for
\[
 (\ell_1,\ell_2,r_1,r_2)
 =(1,2,2,4),(1,2,3,4),(1,3,3,4),
\]
and equals $1$ in all three cases.  It is zero otherwise.  Hence
$H_0(X;S,U)\geq0$ for every endpoint configuration, including sampled
points on knots and cross-family coincidences.

\section{Normalization and sign ledger}\label{app:ledger}

For ease of checking, we collect the constants and signs used in the proof.

\begin{description}
\item[Marsden scalar.]  For normalized degree-$d$ B-splines,
\[
 (y-x)^d=\sum_i
 \left(\prod_{h=1}^{d}(y-T_{i+h})\right)B_{i,d}(x).
\]
Thus \eqref{eq:marsden-factor} has scalar one.  Reversing the $d$ factors
in a right column cancels the reflection sign $(-1)^d$.

\item[Seam sign.]  In Lemma~\ref{lem:small-seam}, the complete exponent is
$b(q-1)+b\ell=b(b-1)+2b\ell$, hence even.

\item[Condensation.]  Desnanot--Jacobi has the minus sign displayed in
\eqref{eq:DJ}; Lemma~\ref{lem:width-one} forces one cross term to vanish,
and $D_{\mathrm{mid}}>0$.

\item[Pl\"ucker promotion.]  The relation \eqref{eq:plucker} has a plus
sign after the columns are written in increasing order.  Its order-$(k-1)$
factor is positive and the order-$k$ terms have smaller dispersion.

\item[Volterra.]  One lift contributes exactly $p^{2m}$.  Ordering the two
labelled integration blocks creates the two step determinants; it creates
neither $m!$ nor $1/m!$ and no cross-block sign.

\item[Beta and de Bruijn.]  The Beta constant is
$c_r=(2r+1)!/(r!)^2$.  The full labelled de Bruijn integral has coefficient
$c_r^m/m!$, the ordered chamber has coefficient $c_r^m$, and the only
grouping sign is $(-1)^{\binom m2}$.

\item[Strictness.]  The positive-measure witness is the central simplex
\eqref{eq:central-simplex}.  The binomial expansion must be centered inside
the separating gap; an expansion about zero need not have termwise positive
Cauchy--Binet summands.
\end{description}

\newpage
\section*{Acknowledgments}

The author is grateful to Professor Fei Yu of Zhejiang University for
pointing him to a collection of historical mathematical problems and
suggesting that he investigate several of them, including the present one;
for reading an earlier version of the proof; and for valuable guidance on its
presentation, structure, and historical context.

The author also thanks Professor Bin Dong and Dr.~Haocheng Ju of Peking
University, as well as Professor Gergely B\'erczi of Aarhus University, whose
platform ICM Conjectures (\url{https://icmconjectures.com/}) has been a
valuable resource: it was there that the author encountered a number of
worthwhile problems, among them the question of Colombo treated here.  Thanks
are due as well to Professor Chao Wu of Zhejiang University, whose
MO platform supplies a
range of scientific datasets that assisted the author in optimizing the models
used in this work.

\par\medskip
Use of AI-assisted tools. During the preparation of
this manuscript and its accompanying Lean~4
formalization, the author used AI-assisted research and coding tools,
including the WuJie AI agent and other AI-agent systems, as well as DeepSeek,
Qwen, Kimi, GPT, and other large language models.  These generative-AI systems
played a substantial role in identifying the proof strategy and producing an
initial proof draft.  The author designed and coordinated the AI-assisted
research workflow, including task decomposition, coordination of multiple AI
systems, and allocation of computational resources.  The author subsequently
checked and revised the mathematical arguments and formalized them in Lean~4,
and takes full responsibility for the mathematical content of the paper.

\section*{Code availability}

A complete Lean~4 formalization of the new odd-exponent theorem and its proof
chain is publicly available at
\url{https://github.com/hkjtsgmc79-boop/colombo-odd-lean}.
The software is released under the Apache License~2.0.
The development is pinned to Lean~4.30.0 and mathlib~4.30.0 (commit
\texttt{c5ea00351c28e24afc9f0f84379aa41082b1188f}).  In the audited build,
\texttt{lake build ColomboGeneralK2} completed all $2817$ jobs.  The
production sources contain no \texttt{sorry}, \texttt{admit}, or
project-specific axioms; the final theorem's axiom audit reports only
\texttt{propext}, \texttt{Classical.choice}, and \texttt{Quot.sound}.

\end{document}